\documentclass[12pt,reqno]{amsart}
\usepackage{amsmath,amssymb}
\usepackage[all]{xy}

\usepackage{hyperref}
\hypersetup{
	colorlinks,
	citecolor=red,
	filecolor=black,
	linkcolor=blue,
	urlcolor=black
}

\theoremstyle{plain}
\newtheorem{thm}{Theorem}[section]

\newtheorem{lmm}[thm]{Lemma}

\numberwithin{equation}{section}

\providecommand \xii{\mathcal{L}}

\providecommand \B{\mathrm{B}}

\begin{document}

\title[Genus one enumerative invariants in del-Pezzo surfaces]{Genus one
enumerative invariants in del-Pezzo surfaces with a fixed complex structure}

\author[I. Biswas]{Indranil Biswas}

\address{School of Mathematics,
Tata Institute of fundamental research, Homi Bhabha road, Mumbai 400005, India}

\email{indranil@math.tifr.res.in}

\author[R. Mukherjee]{Ritwik Mukherjee}

\address{School of Mathematics, Tata Institute of Fundamental
Research, Homi Bhabha Road, Mumbai 400005, India}

\email{ritwikm@math.tifr.res.in}

\author[V. Thakre]{Varun Thakre}

\address{Department of Mathematics, Harish Chandra Research Institute, 
Allahabad 211019, India}

\email{varunthakre@hri.res.in}



\date{}

\begin{abstract}
We obtain a formula for the number of genus one curves with a fixed 
complex structure of a given degree on a del-Pezzo surface that pass 
through an appropriate number of generic points of the surface. This 
enumerative problem is expressed as the difference between the 
symplectic invariant and an intersection number on the moduli space of 
rational curves.

\textit{R\'esum\'e.} \textbf{Invariants \'enum\'eratifs de genre un
avec une structure complexe fix\'ee pour des surfaces de del Pezzo.}
Nous obtenons une formule pour le nombre de courbes de genre un avec une 
structure complexe fix\'ee, de degr\'e donn\'e, et passant par un nombre 
appropri\'e de points g\'en\'eriques de la surface. La solution est exprim\'ee 
comme la diff\'erence entre l'invariant symplectique et un nombre 
d'intersection sur l'espace de modules de courbes rationnelles. 
\end{abstract}

\maketitle

\section{Introduction}\label{introduction}

Enumerative Geometry of rational curves in $\mathbb{P}^2_{\mathbb C}$ is a classical 
question in algebraic geometry. A natural generalization is to ask how many elliptic 
curves, with a fixed $j$-invariant, are there of a given degree that pass through 
the right number of generic points. In \cite{Rahul_genus_1} and \cite{ionel_genus1}, 
using methods of algebraic and symplectic geometry respectively, Pandharipande and 
Ionel obtain a formula for the number of degree $d$ genus one curves with a fixed 
complex structure in $\mathbb{P}^2_{\mathbb{C}}$ that pass through 
$3d-1$ generic points. In this paper, we extend their result to del-Pezzo surfaces. 

Let $X$ be a complex del-Pezzo surface and $\beta \,\in\, H_2(X,\, \mathbb{Z})$
a given homology class. Let $n_{0, \beta}$ denote the number of
rational curves of degree $\beta$
in $X$ that pass through $\delta_{\beta}$ generic points, where $\delta_{\beta} \,:=\,
\langle c_1(TX), \beta \rangle -1$. We prove the following:
\begin{thm}
\label{main_thm}
Let $X$ and $\beta$ be as above. Let
$n_{1, \beta}^{j}$ denote the number of 
elliptic curves with fixed $j$ invariant of degree $\beta$
in $X$ that pass through $\delta_{\beta}$ generic points. Then 
$$
n_{1, \beta}^{j}\, =\, \frac{2 g_{\beta}}{|\textnormal{Aut}(\Sigma_1, j) |} n_{0, 
\beta} \qquad \textnormal{where} \qquad g_{\beta} \,:=\, \frac{\beta\cdot \beta - c_1(TX) \cdot \beta +2}{2}\, ,
$$
$|\textnormal{Aut}(\Sigma_1, j) |$ denotes the number of automorphisms 
of a genus one Riemann surface with fixed $j$ invariant that fixes a point 
and ``$\cdot$'' denotes topological intersection.
\end{thm}

Note that $g_{\beta}$ in Theorem \ref{main_thm} coincides with the genus of a smooth 
degree $\beta$ curve on $X$.
The numbers $n_{0, \beta}$ are computed in 
\cite[p.~29]{KoMa} and \cite[Theorem 3.6]{Pandh_Gott} using a recursive formula. 
When $X\,:=\, \mathbb{P}^2$, our formula for $n_{1, \beta}^j$ is consistent with the 
formula of Pandharipande and Ionel in \cite{Rahul_genus_1} and \cite{ionel_genus1}.
In \cite{KoMa}, the authors actually give a formula to compute the genus 
$0$ Gromov--Witten invariants of the del-Pezzo surfaces, which a priori need not 
be enumerative. 
It is shown in 
\cite[page 63, last paragraph]{Pandh_Gott} that the numbers obtained in 
\cite{KoMa} are actually equal to $n_{0, \beta}$. 

The results of Pandharipande and Ionel generalize the result of P. Aluffi; in 
\cite{P_Aluffi}, he 
computes the number of genus one cubics with a fixed complex structure in 
$\mathbb{P}^2$ through $8$ generic points.  

The problem of enumerating elliptic curves with a fixed $j$-invariant 
has also been studied by tropical geometers. 
In \cite{Kerber_Markwig}, Kerber and Markwig 
compute the number of tropical elliptic curves in $\mathbb{P}^2$ with a fixed $j$-invariant. 
Combined with the correspondence theorem \cite[Theorem A]{Yoav_Len_Raghunathan}, one can conclude that the 
number computed is indeed the same as the number of plane elliptic
curves with a fixed $j$-invariant.
Currently, this question is also being studied for other surfaces. 
In \cite{Yoav_Len_Raghunathan}, 
Len and Ranganathan obtain a formula for the number of elliptic 
curves with a fixed $j$-invariant of a given degree 
for Hirzebruch surfaces, using methods from tropical geometry. 

\section{Enumerative versus symplectic invariant}

We now explain the basic idea to compute $n^j_{1, \beta}$. 
Let $(X, J, \omega)$ be a compact semi-positive symplectic manifold, with 
a compatible almost complex structure $J$ 
of dimension $2m$ and $\beta \,\in\, H_2(X, \mathbb{Z})$ 
be a homology class. Let $k$ be a nonnegative integer such that 
$k+ 2 g \,\geq\, 3$. 
Let $\alpha_1, \cdots, \alpha_k$ and $\gamma_1, \cdots, \gamma_l$ 
be integral homology classes in $H_*(X, \mathbb{Z})$ such that 
\begin{align}
\sum_{i=1}^{k} 2 m - \textnormal{deg}(\alpha_i) +
\sum_{j=1}^{l} (2m -2 - \textnormal{deg}(\gamma_j)) \,=\, 
2m(1-g) + 2 \langle c_1(TX), ~\beta \rangle\, . \label{dim_condition}
\end{align}
Fix pseudocycles $A_i$, $1\,\leq\, i\, \leq\, k$, and $B_j$, $1\,\leq\, j\, \leq\, l$,
on $X$ representing the homology classes $\alpha_i$ and $\gamma_j$. Fix a compact
Riemann surface $\Sigma_g$ of genus $g$; its complex structure will be denoted by $j$.
Define 
$$
\mathcal{M}_{g, k}^{\nu, j} (X, \beta; \alpha_1, \cdots, \alpha_k; \gamma_1,
\cdots, \gamma_k)\,:=\, \{ (u, y_1, \cdots, y_k) \,\in\,
\mathcal{C}^{\infty}(\Sigma_g, X) \times X^k\,\mid ~~u_*[\Sigma_g] = \beta,
$$
$$
\, \overline{\partial}_{j, J} u \,=\, \nu,\ 
u(y_i) \,\in\, A_i ~~\forall ~i = 1, \cdots ,k,\ 
~~\textnormal{Im}(u)\cap \B_j\,\not=\,\emptyset ~~\forall j =1, \cdots ,l \}\, ,
$$
where $\nu\,:\, \Sigma_g \times X\,\longrightarrow\, T^*\Sigma_g \otimes TX$ 
is a generic smooth perturbation and 
\[~\overline{\partial}_{j, J} u := \frac{1}{2} \Big( du + J\circ du \circ j\Big). \]
The symplectic invariant (or the Ruan--Tian invariant) is defined to be the 
signed cardinality of the above set, i.e., 
$$
\mathrm{RT}_{g, \beta}(\alpha_1, \cdots, \alpha_k; \gamma_1, \cdots, \gamma_l)\,:=\, 
\pm |\mathcal{M}_{g, k}^{\nu, j} (X, \beta; \alpha_1, \cdots, \alpha_k; \gamma_1, \cdots,
\gamma_k)|\, . 
$$
When $k=0$, we denote the invariant as 
$$\mathrm{RT}_{g, \beta}(\emptyset ; \gamma_1, \cdots, \gamma_l).$$ 
Furthermore, when $\gamma_1, \ldots, \gamma_l$ all denote the class of a point, 
then we abbreviate the invariant as $\mathrm{RT}_{g, \beta}$. 
Similarly, when $l=0$ we denote the invariant as 
$$\mathrm{RT}_{g, \beta}(\alpha_1, \cdots, \alpha_k; \emptyset).$$
If \eqref{dim_condition} is not satisfied, then 
we formally define the invariant to be zero.  

A natural question is to ask whether the 
symplectic invariant $\mathrm{RT}_{g,\beta}$
is equal to the enumerative invariant $n_{g,\beta}^j$. 
For $\mathbb{P}^2$, and more generally for del-Pezzo surfaces, the 
genus zero symplectic invariant is equal to the enumerative invariant 
\cite[page 267]{RT}. However, even for $\mathbb{P}^2$, the genus 
one symplectic invariant is not enumerative. 
In general, the following fact is true (\cite[Theorem 1.1]{zinger_one_nodal}) 
\begin{equation}\label{rt_equal_eg_plus_cr}
\mathrm{RT}_{1,\beta} \,=\, |\textnormal{Aut}(\Sigma_1, j) | n^j_{1, \beta}
+ \mathrm{CR}\, ,
\end{equation}
where $\mathrm{CR}$ denotes a correction term.  
Let us explain what this term means. 
First we note that the factor of $|\textnormal{Aut}(\Sigma_1, j) |$ is there because we do not 
mod out by automorphisms in the definition of $\mathcal{M}^{\nu, j}_{g,k}$. 
Hence, if 
$u: (\Sigma_1, j) \longrightarrow X$ is a solution to the $\overline\partial$--equation 
and the complex structure on $X$ is genus one regular, 
then there will 
be $|\textnormal{Aut}(\Sigma_1, j) |$ new solutions close to $u$ to the perturbed 
$\overline\partial$--equation. Next, we note that as $\nu \rightarrow 0$, a sequence of
$(J,\nu)$-holomorphic maps can also converge to a bubble tree whose
base (the torus) is a constant (ghost) map \cite[page 2]{ionel_genus1}. 
These maps will also contribute to the computation of $\mathrm{RT}_{1,\beta}$ invariant. 
This extra contribution is defined to be the correction term $\mathrm{CR}$. 

We now explain how to compute the correction term. 
Let $\mathcal{M}_{0,n}(X, \beta)$ denote 
the moduli space of rational degree $\beta$ curves on $X$ that 
represent the class $\beta\,\in\, H_2(X,\, {\mathbb Z})$ 
and are equipped with $n$ ordered marked points, modulo equivalence. 
In other words,
$$
\mathcal{M}_{0,n}(X, \beta)\,:=\, \{ (u, y_1, \cdots, y_n) \,\in\,
\mathcal{C}^{\infty}(\mathbb{P}^1, X) \times
(\mathbb{P}^1)^n\,\mid ~ \overline{\partial}u\,=\,0, ~~u_*[\mathbb{P}^1]
\,= \,\beta \}/\text{PSL} (2, \mathbb{C})\, , 
$$
with $\text{PSL}(2, \mathbb{C})$ acting diagonally on $\mathbb{P}^1\times
(\mathbb{P}^1)^n$. 
Let $\overline{\mathcal{M}}_{0,n}(X, \beta)$ denote the stable map 
compactification of $\mathcal{M}_{0,n}(X, \beta)$. 

Let us now focus on $\overline{\mathcal{M}}_{0,1}(X, \beta)$, the moduli 
space of curves with one marked point. Let $\mathcal{H}$ be the divisor in 
$\overline{\mathcal{M}}_{0,1}(X, \beta)$  
corresponding to the extra condition that the curve
passes through a given point. Let 
$\mathcal{L} \longrightarrow \overline{\mathcal{M}}_{0,1}(X, \beta)$ 
and $\textnormal{ev}: \overline{\mathcal{M}}_{0,1}(X, \beta) \longrightarrow X$
be the universal tangent bundle and the evaluation map 
at the marked point. 
Following the same argument as in
\cite[Lemma 1.23]{ionel_genus1},
we conclude that the bundle 
$\textnormal{ev}^*TX \longrightarrow \overline{\mathcal{M}}_{0,1}(X, \beta) \cap \mathcal{H}^{\delta_{\beta}}$ 
admits a nowhere vanishing section $\nu$. 
This is because the rank of $\textnormal{ev}^*TX$ is two, while the dimension of the 
variety $\overline{\mathcal{M}}_{0,1}(X, \beta) \cap \mathcal{H}^{\delta_{\beta}}$ 
is one. Hence 
$\textnormal{ev}^*TX \longrightarrow \overline{\mathcal{M}}_{0,1}(X, \beta) \cap \mathcal{H}^{\delta_{\beta}}$ 
admits a trivial sub bundle spanned by $\nu$, which we denote 
as $\mathbb{C}\langle \nu \rangle$. When $X:= \mathbb{P}^2$, it is shown in 
\cite[Lemma 1.25]{ionel_genus1}, that 
the correction term is given by
\begin{equation}
\mathrm{CR} \,=\, 
\langle c_1(\mathcal{L}^* \otimes \textnormal{ev}^*TX/ \mathbb{C} \langle \nu \rangle), 
~[\overline{\mathcal{M}}_{0,1}(X, \beta)] \cap \mathcal{H}^{\delta_{\beta}} \rangle. 
\label{cr_formula}
\end{equation}
A more detailed justification of \eqref{cr_formula} is given in \cite{zinger_one_nodal}, by using the 
results of \cite{ZiSG}. 
Furthermore, the gluing construction in \cite{ZiSG} 
is valid in general for K{\"a}hler manifolds \cite[page
8]{ZiSG}. Hence, we conclude that 
\eqref{cr_formula} holds for del-Pezzo surfaces as well. 
Zinger also pointed out this fact to the second author of this paper 
in a personal communication (\cite{ZiCommuinication}).

In the next section we will obtain a formula for $c_1(\mathcal{L}^*)$ and  
compute the right hand side of \eqref{cr_formula}. 
The left hand side of \eqref{rt_equal_eg_plus_cr} is computed using the formula given in 
\cite{RT}. Hence, we obtain $n^j_{1, \beta}$. 

\section{Computation of the correction term} 

We will now give a self contained proof of 
obtaining a formula for  $c_1(\mathcal{L}^*)$ and hence computing the correction term. 
Alternatively, one can also compute the Chern classes by using the dilation equation and the divisor 
equation as given in \cite[Section 26.3]{MirrSymm}. 

\begin{lmm}
\label{c1_divisor_ionel}
On $\overline{\mathcal{M}}_{0,1}(X, \beta)$, the following equality of divisors holds:
\begin{equation}
c_1(\xii^*) \,=\, \frac{1}{(\beta \cdot x_1)^2}
\Big( (x_1 \cdot x_1) \mathcal{H} -2 (\beta \cdot x_1) \textnormal{ev}^*(x_1) +
\sum_{\substack{\beta_1+ \beta_2= \beta, \\ \beta_1, \beta_2 \neq 0}}
\mathcal{B}_{\beta_1, \beta_2} (\beta_2 \cdot x_1)^2 \Big)\, ,\label{chern_class_divisor}
\end{equation}
where $\mathcal{H}$ is the locus satisfying the extra condition that the curve
passes through a given point, $\mathcal{B}_{\beta_1, \beta_2}$ denotes the
boundary stratum corresponding to the splitting into a
degree $\beta_1$ curve and degree $\beta_2$ curve with the last marked point
lying on the degree $\beta_1$ component and $x_i \,:= \,c_i(TX)$.
\end{lmm}

\begin{proof}
The proof is similar to the one given in \cite{ionel_genus1}. Let
$\mu_1\, , \mu_2 \,\in\, X$ be two generic pseudocycles in $X$ that represent the class 
Poincar\'e dual to 
$x_1$.  
Let $\widetilde{\mathcal{M}}$ be a cover of $\overline{\mathcal{M}}_{0,1}(X, \beta)$ 
with two additional
marked points with the last two marked points lying on $\mu_1$ and $\mu_2$
respectively. More precisely,
\begin{align*}
\widetilde{\mathcal{M}} &\,:=\,  \textnormal{ev}_{2}^{-1}(\mu_1) \cap
\textnormal{ev}_{3}^{-1}(\mu_2) \subset
\overline{\mathcal{M}}_{0, 3}(X, \beta)\,
\end{align*}
where $\textnormal{ev}_2$ and $\textnormal{ev}_3$ denote the 
evaluation maps at the second and third marked points respectively. 
Note that the projection $\pi\,:\, \widetilde{\mathcal{M}} \,\longrightarrow\,
\overline{\mathcal{M}}_{0,1}(X, \beta)$ that forgets the last two marked points is a
$(\beta \cdot x_1)^2$--to--one map. We now construct a meromorphic section
\begin{align}
\phi\, :\,\widetilde{\mathcal{M}} \,\longrightarrow\, \pi^*\xii^*  \qquad \textnormal{given by} 
\qquad \phi ([u, y_{_1};
y_{_{2}},
y_{_{3}}]) &:= \frac{{(y_{_2}} - y_{_{3}})
d y_{_{1}}}
{(y_{_{1}}-y_{_{2}})(y_{_{1}}-
y_{_{3}})}. \label{section_ionel}
\end{align}
The right--hand side of \eqref{section_ionel} involves an abuse of notation: it is to be
interpreted in an affine coordinate chart and then extended as a meromorphic section
on the whole of $\mathbb{P}^1$. Note that on $({\mathbb P}^1)^3$, the holomorphic
line bundle
$$
\eta\, :=\,
q^*_1K_{{\mathbb P}^1}\otimes{\mathcal O}_{({\mathbb P}^1)^3}(\Delta_{12}
+\Delta_{13}-\Delta_{23})$$ is trivial, where $q_1\, :\, ({\mathbb P}^1)^3\,
\longrightarrow\,{\mathbb P}^1$ is the projection to the first factor and
$\Delta_{jk}\,\subset\, ({\mathbb P}^1)^3$ is the divisor consisting of all points
$(z_i\, ,z_2\, ,z_3)$ such that $z_j\,=\, z_k$. The diagonal action of ${\rm PSL}(2,
{\mathbb C})$ on $({\mathbb P}^1)^3$ lifts to $\eta$ preserving its
trivialization. The section $\phi$ in \eqref{section_ionel} is given by this trivialization of $\eta$.

Since $c_1(\pi^*\xii^*)$ is the zero divisor minus the
pole divisor of $\phi$, we gather that
\begin{align*}
c_1(\pi^*\xii^*) &\,=\, \{ y_{_{2}} = y_{_{3}}\}
-\{y_{_{1}}=y_{_{2}} \}-
\{y_{_{1}}= y_{_{3}}\}\, .
\end{align*}
When projected down to $\overline{\mathcal{M}}_{0,1}(X,\beta)$, the divisor
$\{ y_{_{2}} = y_{_{3}}\}$ becomes
$$(x_1\cdot x_1)\mathcal{H} + (\beta_2 \cdot x_1)^2 \mathcal{B}_{\beta_1, \beta_2},$$
while both the divisors $\{y_{_{1}}=y_{_{2}} \}$ and
$\{y_{_{1}}=y_{_{3}} \}$ become
$(\beta \cdot x_1) \textnormal{ev}^*(x_1)$.
Since $\widetilde{\mathcal{M}}$
is a $(\beta \cdot x_1)^2$--to--one cover of $\overline{\mathcal{M}}_{0,1}(X,\beta)$, 
we obtain \eqref{chern_class_divisor}.
\end{proof}

Using Lemma \ref{c1_divisor_ionel}, we conclude that 
\begin{equation}
\langle c_1(\mathcal{L}^*), ~[\overline{\mathcal{M}}_{0,1}(X, \beta)] 
\cap \mathcal{H}^{\delta_{\beta}} \rangle \,=\, - 2n_{0, \beta}. 
\label{chern_class_calculation}
\end{equation}
To see why this is so, we first note that 
$\overline{\mathcal{M}}_{0,1}(X, \beta) \cap \mathcal{H}^{\delta_{\beta}+1}$ is zero. 
This is because the number of rational curves through $\delta_{\beta}+1$ generic points is zero. 
Next, we note that 
$\overline{\mathcal{M}}_{0,1}(X, \beta) \cap \mathcal{H}^{\delta_{\beta}} \cap \mathcal{B}_{\beta_1, \beta_2}$ 
is also zero. This is because the number of $\beta$ curves 
which pass through $\delta_{\beta}$ points 
can not split into a degree $\beta_1$ curve and a degree $\beta_2$ curve. This is because such a 
split curve will pass through $\delta_{\beta_1} + \delta_{\beta_2}$ points, which is one less than 
$\delta_{\beta}$. So a split curve can not pass through $\delta_{\beta}$ generic points. 
Finally we note that for any homology class $\mu \in H_2(X, \mathbb{Z})$, the following is true  
\begin{align}
[\overline{\mathcal{M}}_{0,1}(X, \beta)] \cap \mathcal{H}^{\delta_{\beta}}\cap 
\textnormal{ev}^*[\mu] = n_{0,\beta} (\beta \cdot \mu). \label{rt_beta_multiply}
\end{align}
To see why this is so, we note that the left hand side of \eqref{rt_beta_multiply} counts the number of degree 
$\beta$ rational curves through $\delta_{\beta}$ points and one marked point, such that the 
marked point lies on some cycle representing the class $\beta$. There are $\beta \cdot \mu$ choices 
for that marked point to lie, which gives us the right hand side of \eqref{rt_beta_multiply}. These three 
facts give us \eqref{chern_class_calculation}. Note that, when we say $\textnormal{ev}^*[\mu]$, we mean the 
pullback of the cohomology class Poincar\'e dual to $\mu$ (inside $X$). 
Using \eqref{rt_beta_multiply} (with $\mu:= c_1(TX)$), we conclude that  
\begin{equation}
\label{chern_class_calculation_again}
\langle c_1(\textnormal{ev}^*TX), ~[\overline{\mathcal{M}}_{0,1}(X,\beta)] \cap
\mathcal{H}^{\delta_{\beta}} \rangle \,=\, 
\Big(\beta\cdot c_1(TX) \Big) n_{0,\beta}\, .
\end{equation}
From \eqref{chern_class_calculation}, \eqref{chern_class_calculation_again} 
and \eqref{cr_formula} it follows that 
\begin{align}
\mathrm{CR} \,=\, \Big(\beta\cdot c_1(TX)-2\Big) n_{0,\beta}\, . \label{CR_calculation}
\end{align}

\section{Computation of the symplectic invariant}

We now compute the symplectic invariant 
$\mathrm{RT}_{1,\beta}:= \mathrm{RT}_{1,\beta}(\emptyset; p_1, \ldots, p_{\delta_{\beta}})$ 
using the formula
\cite[page 263, (1.2)]{RT}. 
Let $e_1\, , e_2\, , \cdots\, , e_k$ be a basis for $H_*(X,\, \mathbb{Z})$. Let
$$
g_{ij}\,:=\, e_{i} \cdot e_j \ \ \qquad \textnormal{ and }\ \ \qquad g^{ij}\,:=
\,\Big(g^{-1}\Big)_{ij}\, .
$$
If the degrees of $e_i$ and $e_j$ do not add up to be the dimension of $X$ 
then define $g_{ij}$ to be zero. 
Using \cite[page 263, (1.2)]{RT} 
we conclude that 
\begin{align}
\mathrm{RT}_{1,\beta}(\emptyset; p_1, \ldots, p_{\delta_{\beta}}) =  
\sum_{i,j} g^{ij} \mathrm{RT}_{0,\beta}(e_i, e_j; p_1, \ldots, p_{\delta_{\beta}}) 
& = \sum_{i,j} g^{ij} n_{0, \beta} (\beta \cdot e_i)(\beta \cdot e_j) \nonumber \\ 
& = (\beta \cdot \beta) n_{0, \beta}.\, \label{RT_calculation}
\end{align}
The last equality follows by writing $\beta$ in the given basis $e_i$ and using the 
definition of $g^{ij}$; the second equality follows from the same we justify \eqref{rt_beta_multiply}. 
Equations \eqref{RT_calculation}, \eqref{CR_calculation} and \eqref{rt_equal_eg_plus_cr} give 
us the formula of Theorem \ref{main_thm}.

\section{Regularity of the complex structure for del-Pezzo surfaces}
We now show that the complex structure on the del-Pezzo surfaces is genus one regular 
for immersion. In the statement of Theorem \ref{main_thm} the curve $u$ passes 
passes through $\delta_{\beta}$ generic points. Hence the curve is going to be 
an immersion and hence it suffices to prove regularity for immersions. 

\begin{lmm}
Let $X$ be $\mathbb{P}^2$ blown up at $k$ points and 
$(\Sigma_1, j)$ a compact genus $1$ Riemann surface with a 
complex structure $j$. 
Let 
$u: \Sigma_1 \longrightarrow X$ be a holomorphic map representing the 
class $\beta := d L - m_1 E_1- \ldots -m_k E_k \in H_2(X, \mathbb{Z})$, 
where $L$ and $E_i$ denote the class of a line and the exceptional divisors 
respectively. 
Then $D_u$, the linearization of the $\overline{\partial}_{j,J}$ at $u$ is surjective, 
provided $d >0$ and $u$ is an immersion. 
In particular, the complex structure on the del-Pezzo surface 
is genus $1$ regular for immersions.  
\end{lmm}

\begin{proof}
We first note that if ${\mathcal L}$ is a holomorphic line bundle on $(\Sigma_1, j)$ of
positive degree, then the cup product
$$H^0(\Sigma_1, {\mathcal L})\otimes H^1(\Sigma_1, {\mathcal L}^*)
\,\longrightarrow\, H^1(\Sigma_1, \mathcal L)\otimes {\mathcal L}^*)\,=\, \mathbb C
$$
is nondegenerate. Indeed, this coincides with the Serre duality pairing because the
canonical line bundle of $\Sigma_1$ is trivial, and hence the pairing is
nondegenerate. Next consider the short exact sequence of vector bundles on $\Sigma_1$
given by the differential of $u$
\begin{equation}\label{g1}
0\,\longrightarrow\, T\Sigma_1\,\stackrel{du}{\longrightarrow}\,u^*TX
\,\longrightarrow\,  Q\,:=\, (u^*TX)/T\Sigma_1\,\longrightarrow\, 0\,.
\end{equation}
Let
\begin{equation}\label{g2}
H^0(\Sigma_1, Q)\,\stackrel{\rho}{\longrightarrow}\,H^1(\Sigma_1, T\Sigma_1)
\,\longrightarrow\,H^1(\Sigma_1, u^*TX) \,\longrightarrow\,H^1(\Sigma_1, Q)
\end{equation}
be the long exact sequence of cohomologies associated to it. We have
$H^1(\Sigma_1, Q)\,=\, 0$ because $\text{degree}(Q)\, >\, 0$ (note that
$\text{degree}(u^*TX)\, >\, 0$ and $\text{degree}(T\Sigma_1)\,=\, 0$). The exact
sequence in \eqref{g1} does not split; for the corresponding extension class $\psi\, \in\,
H^1(\Sigma_1, (T\Sigma_1)\otimes Q^*)\,=\, H^1(\Sigma_1, Q^*)$, as observed before,
there is $\psi'\, \in\, H^0(\Sigma_1, Q)$ such that $\psi\cup \psi'\,\not=\, 0$. Hence
$\rho$ in \eqref{g2} is nonzero. This implies that $\rho$ is surjective because
$\dim H^1(\Sigma_1, u^*TX)\,=\, 1$. hence from \eqref{g2} we conclude that
$H^1(\Sigma_1, u^*TX)\,=\, 0$, which proves that 
the cokernel of $D_u$ is zero (i.e., $D_u$ is surjective).
\end{proof}

When $X:= \mathbb{P}^1 \times \mathbb{P}^1$, the complex structure is genus one regular; that is 
because for $\mathbb{P}^1$ the complex structure is genus one regular (by \cite[Corollary 6.5]{g2p2and3}). 
Hence, the same fact holds for products of $\mathbb{P}^1$. 

\section*{Acknowledgements}

We are very grateful to the two referees for detailed comments.
The second author is grateful to Hannah Markwig and Yoav Len for some fruitful 
discussions on this subject and informing us about the ongoing work 
\cite{Yoav_Len_Raghunathan}. We thank them for letting us know that this 
question is of interest to tropical geometers.

\end{document}